\documentclass[reqno,11pt]{amsart}
\usepackage{amsmath,amssymb,amsfonts,amsthm,amscd,indentfirst,thmtools,mathtools,mathrsfs}

\usepackage[hyphens]{url}
\usepackage{hyperref}
\usepackage{bookmark}
\mathtoolsset{showonlyrefs=true}

\newcounter{mparcnt}

\usepackage{fancyhdr}
\usepackage{esint}
\usepackage{enumerate}

\usepackage{pictexwd,dcpic}
\usepackage{graphicx}
\usepackage{caption}
\usepackage{xcolor}
\usepackage{enumitem}

\usepackage[msc-links]{amsrefs} 

\usepackage{epsfig,here}
\usepackage{subfigure,here}

\newtheorem{theorem}{Theorem}[section]
\newtheorem{lemma}[theorem]{Lemma}
\newtheorem{proposition}[theorem]{Proposition}
\newtheorem{definition}[theorem]{Definition}
\newtheorem{remark}[theorem]{Remark}

\newtheorem{thmx}{Theorem}

\usepackage[top=4cm, bottom=4cm, right=3cm, left=3cm]{geometry}

\def\om{\omega}
\def\Om{\Omega}
\def\p{\partial}

\def\De{\Delta}

\def\S{{\Sigma}}
\def\<{\langle}
\def\>{\rangle}

\def\na{\nabla}

\providecommand{\abs}[1]{\lvert#1\rvert}
\providecommand{\Abs}[1]{\left\lvert#1\right\rvert}

\newcommand{\mbB}{\mathbb{B}}

\newcommand{\mbN}{\mathbb{N}}

\newcommand{\mbR}{\mathbb{R}}
\newcommand{\mbS}{\mathbb{S}}

\newcommand{\mcC}{\mathcal{C}}

\newcommand{\mcH}{\mathcal{H}}

\newcommand{\mcL}{\mathcal{L}}

\newcommand{\mcW}{\mathcal{W}}

\newcommand{\mfH}{\mathbf{H}}

\newcommand{\mfR}{\mathbf{R}}

\newcommand{\rd}{{\rm d}}

\newcommand{\ra}{\rightarrow}

\newcommand{\eq}[1]{\begin{equation}\begin{alignedat}{2} #1 \end{alignedat}\end{equation}}

\numberwithin{equation} {section}

\begin{document}
	
	\title[Willmore inequality]{Willmore-type inequality in unbounded convex sets}

\author[Jia]{Xiaohan Jia}
	\address[X.J]{School of Mathematics\\
Southeast University\\
 211189, Nanjing, P.R. China}
	\email{xhjia@seu.edu.cn}
    
	
\author[Wang]{Guofang Wang}
	\address[G.W]{Mathematisches Institut\\
		Universit\"at Freiburg\\
	Ernst-Zermelo-Str.1\\
		79104\\
  \newline\indent Freiburg\\ Germany}
	\email{guofang.wang@math.uni-freiburg.de}

 
\author[Xia]{Chao Xia}
	\address[C.X]{School of Mathematical Sciences\\
		Xiamen University\\
		361005\\ Xiamen\\ P.R. China}
	\email{chaoxia@xmu.edu.cn}


\author[Zhang]{Xuwen Zhang}
	\address[X.Z]{Mathematisches Institut\\
		Universit\"at Freiburg\\
	Ernst-Zermelo-Str.1\\
		79104\\
  \newline\indent Freiburg\\ Germany}
\email{xuwen.zhang@math.uni-freiburg.de}

\begin{abstract}

In this paper we prove the following Willmore-type inequality: On an unbounded closed convex set $K\subset\mbR^{n+1}$ $( n\ge 2$), for any embedded hypersurface $\S\subset K$ with boundary $\p\S\subset \p K$ satisfying a certain contact angle condition, there holds
\eq{
{\frac1{n+1}\int_{\S}\Abs{{H}}^n \rd A\ge \rm AVR}(K)\abs{\mbB^{n+1}}.
}
Moreover, equality holds if and only if $\S$ a part of a sphere and  $K\setminus \Om$ is a part of the solid cone determined by $\S$. Here $\Om$ is the bounded domain enclosed by $\S$ and $\p K$, $H$ is the normalized
mean curvature of $\S$, and ${\rm AVR}(K)$ is the asymptotic volume ratio of $K$.
We also prove an anisotropic version of this Willmore-type inequality. 
As a special case, we obtain a Willmore-type inequality for anisotropic capillary hypersurfaces in a half-space.

		\
		
  \\
		
	\end{abstract}

\maketitle

\section{Introduction}

The classical Willmore inequality  \cites{Willmore68,Chen71} states that for a bounded domain $\Om\subset\mbR^{n+1}$ with smooth boundary, there holds
\eq{\label{ineq-Willmore-classical}
\frac1{n+1}\int_{\p\Om}\Abs{H}^n \rd A\ge \abs{\mbB^{n+1}}
,
}
where $H$ is the normalized mean curvature of $\p\Om$ (i.e. $H=\frac{{\rm tr}B}n$ where $B$ is the second fundamental form of $\p\Om\subset\mfR^{n+1}$), and $\abs{\mbB^{n+1}}$ is the volume of unit ball $\mbB^{n+1}$.
Moreover, equality in \eqref{ineq-Willmore-classical} holds if and only if $\Om$ is a round ball.

Recently, Agostiniani-Fogagnolo-Mazzieri \cite{AFM20} proved the following Willmore-type inequality in Riemannian manifolds with nonnegative Ricci curvature.
\begin{thmx}[\cite{AFM20}*{Theorem 1.1}]\label{Thm-AFM20-Thm1.1}
Let $(M^{n+1},g)$  $(n\geq2)$ be a complete Riemannian manifold with nonnegative Ricci curvature and $\Om\subset M$ a bounded open set with smooth boundary.
Then
\eq{
 \frac 1 {n+1}\int_{\p\Om}
\Abs{H}^n \rd A\ge {\rm AVR}(g)\abs{\mbB^{n+1}},
}
where ${\rm AVR}(g)$ is the asymptotic volume ratio of M 
(i.e. the limit $\lim_{r\ra\infty}\frac{(n+1)\abs{B_r(p)}}{r^{n+1}\abs{\mbS^n}}$, which exists thanks to the Bishop-Gromov volume comparison theorem).
Moreover, if ${\rm AVR}(g)>0$, equality holds if and only if $M\setminus\Om$ is isometric to $([r_0,\infty)\times\p\Om,\rd r^2+\left(\frac{r}{r_0}g_{\p\Om})^2\right)$ with
\eq{
r_0
=\left(\frac{\abs{\p\Om}}{{\rm AVR}(g)\abs{\mbS^n}}\right)^\frac1n.
}
\end{thmx}
Soon after Wang \cite{Wang23} gave a short proof of Theorem \ref{Thm-AFM20-Thm1.1}, which is based on Heintze-Karcher's comparison theorem in Riemannian geometry.

In this paper, we study a similar problem in \textit{Euclidean unbounded closed convex sets}. Let us first introduce some terminologies and properties of unbounded closed convex sets, which are needed to state our main result.
Let $K\subset\mbR^{n+1}$ $( n\ge2)$ be an unbounded closed convex set.
We denote by ${\rm Reg}(\p K)$ the regular part of $\p K$, that is, the set of points near which $\p K$ can be locally written as a $C^1$-hypersurface, and  ${\rm Sing}(\p K)=\p K\setminus{\rm Reg}(\p K)$ the singular part of $\p K$.
For $x\in {\rm Reg}(\p K)$, we denote by $\bar N(x)$ the outward unit normal to $\p K$ at $x$. Let $B_R(x)$ denote the open ball of radius $R$ centered at $x$.

We define
\eq{
{\rm AVR}(K)
=\lim_{R\to\infty} \frac{\vert B_R(0)\cap K\vert}{\abs{\mbB^{n+1}}R^{n+1}}.
}
To see that ${\rm AVR}(K)$ and its generalization are well-defined, see Section \ref{App-B} below. Moreover,
It is not difficult to check  that 
\eq{
{\rm AVR}(K)
=\lim_{R\to\infty} \frac{\vert B_R(p)\cap K\vert}{\abs{\mbB^{n+1}}R^{n+1}},
}
for any $p\in \mbR^{n+1}$.

Our main result in this paper is the following Willmore-type inequality in unbounded convex sets.

\begin{theorem}\label{Thm-iso-convex}
Let $K\subset \mbR^{n+1}$ be an unbounded closed convex set. 
Let $\S\subset{K}$ be a compact, embedded $C^2$-hypersurface with boundary $\p\S\subset{\rm Reg}(\p K)$ intersecting $\p K$ transversally such that
\eq{
	\<\nu(x),\bar N(x)\>\geq 0,\quad \hbox{ for any } x\in\p\S.
}
Here $\nu$ denotes the outward unit normal to $\S$ with respect to $\Om$, the bounded domain enclosed by $\S$ and $\p K$.
Then there holds
\eq{\label{ineq-willmore-iso-cone}
\frac{1}{n+1}\int_\S \Abs{H}^{n} \rd A\ge {\rm AVR}(K)\abs{\mbB^{n+1}}.
}
Moreover, equality in \eqref{ineq-willmore-iso-cone} holds if and only if $\S$ is a part of a sphere and $K\setminus \Om$ is a part of the solid cone determined by $\S$.

\end{theorem}

Our theorem can be viewed as an extrinsic counterpart of Theorem \ref{Thm-AFM20-Thm1.1}.
The approach to this theorem is inspired by Wang's short proof \cite{Wang23} of Theorem \ref{Thm-AFM20-Thm1.1} based on Heintze-Karcher's comparison \cite{HK78}, and also inspired by our recent works on Heintze-Karcher-type inequalities in various circumstances, see \cites{JWXZ22,JWXZ23,JWXZ23b}.
We refer the interested reader to \cite{JWXZ23b}*{Theorem 1.2} for the Heintze-Karcher-type inequality in arbitrary convex sets, which could be viewed as a "dual" version of the above Willmore inequality in unbounded convex sets.

{We would like to call attention to a result by Choe-Ghomi-Ritore \cite{CGR06}, which says that for a compact free boundary hypersurface $\S\subset\mbR^{n+1}$ {\it outside a convex set,} there holds
\eq{
\label{eq1.3} \frac{1}{n+1}\int_\S \Abs{H}^n dA\ge \frac12 \abs{\mbB^{n+1}}} with equality if and only if $\S$ is a hemisphere lying in a half-space.} This inequality leads to an optimal relative  isoperimetric inequality in \cite{CGR07}. See also \cite{LWW23}. 
In contrast all our results in this paper hold for compact hypersurfaces {\it inside a convex set}.

In view of our previous work \cite{JWXZ23b}, it is not surprising that we could in fact establish the following Willmore-type inequality in unbounded convex sets, with anisotropy taken into account. Here anisotropy means a smooth positive function $F: \mathbb{S}^n\to \mathbb{R}_+$ on the unit sphere $(\mathbb{S}^n,\sigma)$ such that 
$(\nabla^2 F+F \sigma)
$ is positive definite.
In literature $F$ is usually called a Minkowski norm. ($F$ induces a norm if $F$ is even, i.e, $F(-x)=F(x)$. In this paper we do not require the evenness.)
With respect to a Minkowski norm $F$,
one can define the 
(unit) Wulff ball $\mcW^F$, the  asymptotic volume ratio ${\rm AVR}_F$ for unbounded convex sets, the anisotropic unit normal $\nu_F$ and the normalized anisotropic mean curvature $H^F$ for a hypersurface. For more definitions and notation see
Section \ref{Sec-2}, in particular Section \ref{App-B}.
\begin{theorem}\label{Thm-aniso-convex}
Let $K\subset \mbR^{n+1}$ be an unbounded closed convex set. 
Let $\S\subset{K}$ be a compact, embedded $C^2$-hypersurface with boundary $\p\S\subset{\rm Reg}(\p K)$ intersecting $\p K$ transversally such that
\eq{\label{condi-convexcone-ani}
	\<\nu_F(x),\bar N(x)\>\geq 0,\quad \hbox{ for any } x\in\p\S.
}
{Here $\nu_F(x)=DF(\nu(x))$, and $\nu$ is the outward unit normal to $\S$ with respect to $\Om$, the bounded domain enclosed by $\S$ and $\p K$.}
Then there holds
\eq{\label{ineq-willmore-aniso-cone}
\frac{1}{n+1}\int_\S F(\nu)\abs{{H^F}}^{n} \rd A\ge {\rm AVR}_F(K)\abs{\mcW^F}.
}
Moreover, equality in \eqref{ineq-willmore-aniso-cone} holds if and only if $\S$ is a part of a Wulff shape {centered at some $x_0\in\mfR^{n+1}$}, and $K\setminus\Om$ is the solid cone {over $\S$ with center $x_0$}.
\end{theorem}
Note that Theorem \ref{Thm-iso-convex} is a special case of Theorem \ref{Thm-aniso-convex} with $F(\xi)=\abs{\xi}$.

{The Willmore-type inequality \eqref{ineq-willmore-aniso-cone} may be applied to prove relative isoperimetric inequality in unbounded convex sets, as Choe-Ghomi-Ritore \cite{CGR07} did outside convex sets. See also \cite{FM23}. We remark that the relative isoperimetric-type inequality in unbounded convex sets in the isotropic case, i.e. $F(\xi)=|\xi|$, has been proved by  Leonardi-Ritoré-Vernadakis \cite{LRV22}.}

{Let $\{E_i\}_{1\leq i\leq (n+1)}$ be the coordinate basis of $\mfR^{n+1}$.
}
Consider the particular case when $K=\overline{\mbR^{n+1}_+}{\coloneqq\{x\in\mfR^{n+1}:\left<x,E_{n+1}\right>=x_{n+1}\geq0\}}$ is the (closed) upper half-space, by choosing $F(\xi)=|\xi|-\cos \theta_0 \<\xi, E_{n+1}\>$ for some $\theta_0\in (0, \pi)$, we get from Theorem \ref{Thm-aniso-convex} the following Willmore-type inequality for capillary hypersurfaces.
\begin{theorem}\label{Thm-Willore-halfspace-iso}
Given $\theta_0\in(0,\pi)$.
Let  $\S\subset\overline{\mbR_+^{n+1}}$ be a compact, embedded $C^2$-hypersurface with boundary $\p\S\subset \p\mbR_+^{n+1}$ intersecting $\p\mbR^{n+1}_+$ transversally such that
\eq{
\left<\nu, -E_{n+1}\right>\geq -\cos\theta_0,\quad \hbox{ for any } x\in\p\S.
}
Then there holds
\eq{\label{ineq-willmore-isotropic-halfspace-intro}
\frac{1}{n+1}\int_{\Sigma}\left(1-\cos\theta_0\<\nu, E_{n+1}\>\right)\Abs{{H}}^{n} \rd A\ge \abs{B_{1,\theta_0}},
}
where  
$B_{1,\theta_0}=B_1(-\cos\theta_0 E_{n+1})\cap {\mbR_+^{n+1}}$ {is the unit ball with center $-\cos\theta_0 E_{n+1}$ truncated by the upper half-space}.
Moreover, equality in \eqref{ineq-willmore-isotropic-halfspace-intro} holds if and only if $\S$ is a $\theta_0$-capillary spherical cap in $\overline{\mbR_+^{n+1}}$.
\end{theorem}
More generally, by choosing $F(\xi)$ in Theorem \ref{Thm-aniso-convex} to be $F(\xi)+\om_0\left<\xi,E^F_{n+1}\right>$ for some $\omega_0$, we get the following Willmore-type inequality for anisotropic capillary hypersurfaces.

\begin{theorem}
Given $\om_0\in\left(-F(E_{n+1}),F(-E_{n+1})\right)$.
Let  $\S\subset\overline{\mbR^{n+1}_+}$ be a compact, embedded $C^2$-hypersurface with boundary $\p\S\subset \p\mbR_+^{n+1}$ intersecting $\p\mbR^{n+1}_+$ transversally such that
\eq{
  \<\nu_F(x), -E_{n+1}\>=\om(x)\geq \om_0,\quad  \hbox{ for  any }x\in \partial\Sigma.
}
Then there holds
\eq{\label{ineq-willmore-ani-halfspace-intro}
\frac1{n+1}\int_\S \left(F(\nu)+\om_0\left<\nu,E^F_{n+1}\right>\right)\abs{{H^F}}^n\rd A\ge \abs{\mcW^F_{1,\om_0}},
}
where $\mcW^F_{1,\om_0}=\mcW^F_1(\om_0 E^F_{n+1})\cap {\mbR_+^{n+1}}$, {and $E_{n+1}^F$ is the constant vector defined as \eqref{eq-E-F}}.
Moreover, equality in \eqref{ineq-willmore-ani-halfspace-intro} holds if and only if $\S$ is an anisotropic $\om_0$-capillary Wulff shape in $\overline{\mbR_+^{n+1}}$.
\end{theorem}

In fact, as already observed in De Phillipis-Maggi's work \cite{DePM15}, anisotropic capillary problems with respect to $F$ can be regarded as anisotropic free boundary problems with respect to another Minkowski norm $\tilde{F}$ in the half-space.
Hypersurfaces with the aforementioned special boundary conditions naturally arise from Calculus of Variations and are nowadays of particular interest.
We refer the readers to \cite{JWXZ23} and the references therein for a short historical introduction.

For the specialty of the half-space, we could in fact prove Theorem \ref{Thm-aniso-convex} in the half-space case in an alternative way.
This is done by a geometric observation on the Gauss image of $\S$, stated in Proposition \ref{Prop-Gauss-map}.
Once this is established, the rest of the proof of the Willmore inequality in the half-space then follows from the classical argument based on an area formula.
{We also mention that by a similar approach, we get another version of the Willmore-type inequality, which we state in Theorem \ref{Thm-aniso-halfspace-'} (see in particular \eqref{ineq-willmore-aniso-halfspace-'}).
This version has been obtained by Fusco-Morini \cite{FM23}*{Theorem 3.9} in the isotropic case. }

\

\subsection*{Outline of the paper}
The rest of the paper is organized as follows. 
In Section \ref{Sec-2}, we collect basic knowledges on Minkowski norm, anisotropic geometry, and anisotropic capillary hypersurfaces.
In Section \ref{Sec-3}, we prove Theorem \ref{Thm-iso-convex}.
In Section \ref{Sec-4}, we give an alternative proof for Theorem \ref{Thm-iso-convex} in the half-space case and then apply it to anisotropic capillary hypersurfaces.

\

\subsection*{Acknowledgements}
We thank Prof. Nicola Fusco and Prof. Massimiliano Morini for informing their work \cite{FM23} on a similar Willmore-type inequality which holds \textit{outside} unbounded convex sets.
We also thank the referee for the comments which help improve the exposition of the paper.

C.X. is supported by NSFC (Grant No. 12271449, 12126102) and the Natural Science Foundation of Fujian Province of China (Grant No. 2024J011008).
 X.J. is supported by National Natural Science Foundation of China
(Grant No. 12401249), Natural Science Foundation of Jiangsu Province, China (Grant No. BK20241258).

\


\section{Preliminaries}\label{Sec-2}


\subsection*{Notations}

The Euclidean metric, scalar product, and Levi-Civita connection of the Euclidean space $\mbR^{n+1}$ are denoted respectively by $g_{\rm euc},\left<\cdot,\cdot\right>$, and $D$.
When considering the topology of $\mbR^{n+1}$, we adopt the following notations:
for a set $E\subset\mbR^{n+1}$,
we denote
by $\overline E$ the topological closure of $E$, by ${\rm int}(E)$ the topological interior of $E$, and by $\p E$ the topological boundary of $E$ in $\mbR^{n+1}$.
Regarding the use of the symbol $\abs{\cdot}$,
if we plug in a vector $e\in\mbR^{n+1},$ then $\abs{e}$ denotes the Euclidean length of $e$.
If we plug in a $k$-dimensional submanifold $M\subset\mbR^{n+1}$, then we write
\eq{
\abs{M}
\coloneqq\mathcal{H}^k(M),
}
where $\mathcal{H}^k$ is the $k$-dimensional Hausdorff measure in $\mbR^{n+1}$.
In particular, if we plug in an open set $\Om$ of $\mbR^{n+1}$, then we mean
\eq{
\abs{\Om}
\coloneqq\mathcal{L}^{n+1}(\Om).
}
To avoid ambiguity, we also use ${\rm Vol}(\cdot)$ to denote the $\mcL^{n+1}$-measure of certain sets in $\mbR^{n+1}$, see for example \eqref{eq-avr-halfspace-ani}.


\subsection{Minkowski norm and anisotropic geometry}
Let $F: \mathbb{S}^n\to \mathbb{R}_+$ be a smooth positive function on the standard sphere $(\mathbb{S}^n,\sigma)$ such that 
\eq{\label{convexity condition}
{
A_F{\mid_x}[z]
\coloneqq\nabla^2 F{\mid_x}[z]+F(x)z,\quad\text{ for every }x\in\mbS^n, z\in T_x\mbS^n
}
}
{is positive definite},
where $\na$ denotes the Levi-Civita connection associated to $\sigma$.
A \textit{Minkowski norm} is the one homogeneous extension of any such $F$ to the whole $\mbR^{n+1}$, namely,
$F(\xi)=|\xi|F(\frac{\xi}{|\xi|})$ for $\xi\neq 0$ and $F(0)=0$.
Note that condition \eqref{convexity condition} is equivalent to saying that $\frac12F^2$ is uniformly convex, in the sense that
\eq{
D^2(\frac12F^2)(\xi)>0,\quad\forall \xi\in\mbR^{n+1}{\setminus\{0\}}.
}

Let $\Phi$ be the \textit{Cahn-Hoffman map} associated to $F$, which is given by
\eq{
\Phi(z)=DF(z)=\nabla F(z)+F(z)z, \quad \forall z \in \mathbb{S}^n.
}
The image $\Phi(\mathbb{S}^n)$ of $\Phi$  is called (unit) Wulff shape.
The \textit{dual Minkowski norm} of $F$, denoted by $F^o$, is given by
\eq{\label{defn-dual}
    F^o(x)=\sup\left\{\frac{\left<x,z\right>}{F(z)}\Big| \, z\in\mbS^n\right\}.
}

We collect some well-known facts on $F$, $F^o$, and $\Phi$, see e.g., \cite{JWXZ23}*{Proposition 2.1}.
\begin{proposition}\label{basic-F} For any $z\in \mbR^{n+1}\setminus\{0\}$
the following statements hold for an Minkowski norm.
\begin{itemize}
    \item[(i)] $F^o(tz)= t F^o(z)$ , for any $t>0$.
    \item[(ii)] $F^o(x+y)\leq F^o(x)+F^o(y)$, for $x, y \in \mbR^{n+1}$.
\item[(iii)] $\left<\Phi(z),z\right>=F(z)$.
\item[(iv)] $F^o(\Phi(z))=1$.
\item[(v)] The following Cauchy-Schwarz inequality holds:
\begin{align}\label{Cau-Sch}\<z, \xi\>\le F^o(z)F(\xi), \quad \forall \xi  \in \mbR^{n+1}.
\end{align}
\item[(vi)]  The  Wulff shape $\Phi(\mathbb{S}^n)=\{x\in\mbR^{n+1}|F^o(x)=1\}$.
\end{itemize}
\end{proposition}
We denote the Wulff ball of radius $r$ and centered at $x_0\in\mbR^{n+1}$  by
\eq{
\mcW_r^F(x_0)
=\{x\in\mbR^{n+1}|F^o(x-x_0)<r\},
}
and the corresponding Wulff shape $\p\mcW_r^F(x_0)$ is given by
\eq{
\p\mcW_r^F(x_0)
=\{x\in\mbR^{n+1}|F^o(x-x_0)=r\}.
}
We also use $\mcW^F$ and $\mcW^F_R$ to abbreviate  $\mcW^F_1(0)$ and $\mcW^F_R(0)$.

Let $\S\subset\mbR^{n+1}$ be a $C^2$-hypersurface and $\nu$  a unit normal field of $\S$.
The \textit{anisotropic normal} of $\S$ with respect to $\nu$  and $F$ is given by
{(see e.g., \cite{Winklmann06})
}
\eq{
\nu_F=\Phi(\nu)=\nabla F(\nu)+F(\nu)\nu,
}
and the \textit{anisotropic principal curvatures} $\{\kappa_i^F\}_{i=1}^n$ of $\S$ with respect to $\nu$ and $F$ are given by the eigenvalues of the \textit{anisotropic Weingarten map} $$\rd\nu_F=A_F(\nu)\circ\rd\nu: T_p\S\to T_p\S.$$ The eigenvalues are real since $A_F$ is positive definite and symmetric.
Let $$H^F=\frac 1 n \sum\limits_{i=1}^n\kappa_i^F \quad \hbox{ and } \quad H^F_n=\prod\limits_{i=1}^n\kappa_i^F$$ denote the normalized \textit{anisotropic mean curvature} and the \textit{anisotropic Gauss-Kronecker curvature} of $\Sigma$ respectively.
It is easy to check that the anisotropic principal curvatures of $\p\mcW^F_r(x_0)$ are $\frac1r$, since
\eq{\label{normal}
    \nu_F(x)=\frac{x-x_0}{r},  \quad\hbox{ on } \p\mcW^F_r(x_0).
}
{For a proof of this fact, see e.g., \cite{JWXZ23}*{(2.3)}.
}
We record
the following very useful anisotropic angle comparison principle.
\begin{proposition}[{\cite{JWXZ23}*{Proposition 3.1}}]\label{Prop-angle-compare}
Let $x,z\in\mbS^n$ be two distinct points and $y\in\mbS^n$ lie in a length-minimizing geodesic joining $x$ and $z$ in $\mbS^n$,
then we have \eq{
\left<\Phi(x),z\right>\leq\left<\Phi(y),z\right>.
}
Equality holds if and only if $x=y$.
\end{proposition}

\subsection{Anisotropic capillary hypersurfaces in a half-space}

Let us first recall the definition of anisotropic capillary hypersurface in the half-space.
\begin{definition}
\normalfont
Given $\om_0\in\left(-F(E_{n+1}),F(-E_{n+1})\right)$, any hypersurface $\S\subset\overline{\mbR^{n+1}_+}$ is said to be \textit{anisotropic $\om_0$-capillary} in $\overline{\mbR^{n+1}_+}$ if it intersects $\p\mbR^{n+1}_+$ transversally with
\eq{
\left<\nu_F(x),-E_{n+1}\right>\equiv\om_0,\quad\forall x\in\p\S.
}
In particular, $\S$ is said to have \textit{anisotropic free boundary}  if it is anisotropic $\om_0$-capillary with $\om_0=0$.
\end{definition}
Here we record some facts concerning anisotropic capillary hypersurfaces in the half-space.

\begin{proposition}[\cite{JWXZ23}*{Remark 2.1}]\label{Prop-JWXZ23-Rem2.1}
Let $\S$ be a hypersurface in $\overline{\mbR^{n+1}_+}$ which meets $\p\mbR^{n+1}_+$ transversally
and define a function on $\partial\Sigma$ by $\om(x)\coloneqq\left<\nu_F(x),-E_{n+1}\right>$. Then for any $x\in\p\S$,
\eq{
\om(x)\in\left(-F(E_{n+1}),F(-E_{n+1})\right).
}
\end{proposition}

Define a constant vector $E^{F}_{n+1}\in\mathbb{R}^{n+1}$ as in \cite{JWXZ23},
\eq{\label{eq-E-F}
E^{F}_{n+1}
=\left\{
\begin{array}{lrl}  -\frac{\Phi(-E_{n+1})}{F(-E_{n+1})}, \quad &\text{if } \om_0>0,\\
\frac{\Phi(E_{n+1})}{F(E_{n+1})}, \quad &\text{if } \om_0<0,\\
 E_{n+1},\quad &\text{if } \om_0=0,\\
  %
\end{array}
\right.
}
whose definition is strongly related to the Cauchy-Schwarz inequality.
When $\om_0=0$, one can also define $E_{n+1}^{F}$ as $\frac{\Phi(E_{n+1})}{F(E_{n+1})}$ or $-\frac{\Phi(-E_{n+1})}{F(-E_{n+1})}$.
Note that $\<E^{F}_{n+1}, E_{n+1}\>=1$, and when $F$ is the Euclidean norm, $E^{F}_{n+1}$ is indeed $E_{n+1}$.

\begin{proposition}[{\cite{JWXZ23}*{Proposition 3.2}}]\label{Prop-non-negative-aniso-halfspace}
For $\om_0\in\left(-F(E_{n+1}),F(-E_{n+1})\right)$, there holds
\begin{align}\label{ineq-non-negative-aniso-halfspace}
F(z)+\om_0\<z,E^F_{n+1}\>>0,\quad\text{for any }z\in\mbS^n.
\end{align}
\end{proposition}

We have the following integral formula.
\begin{lemma}
Let $\S$ be a compact, embedded $C^2$-hypersurface in $\overline{\mbR^{n+1}_+}$ which intersects $\p\mbR^{n+1}_+$ transversally.
Denote by $\Om$ the bounded domain enclosed by $\S$ and $\p \mbR^{n+1}_+$.
Then
there holds
\eq{\label{formu-om_0-E^F}
\int_\S\<\nu,E_{n+1}^F\> \rd A
=\int_{\p\Om\cap\p\mbR^{n+1}_+} \rd A.
}
\end{lemma}
\begin{proof}
Notice that ${\rm div}\left(E_{n+1}^F\right)=0$.
Integrating this over $\Om$, using integration by parts, then invoking the fact that $\left<E_{n+1},E^F_{n+1}\right>=1$, we obtain the assertion.
\end{proof}

As mentioned in the introduction, we shall interpret the anisotropic capillary problem as an anisotropic free boundary problem.
This is done by introducing the following Minkowski norm.
\begin{proposition}\label{Prop-tilde-F}
Given $\om_0\in\left(-F(E_{n+1}),F(-E_{n+1})\right)$.
Let $\S$ be a compact, embedded $C^2$ $\om_0$-capillary hypersurface in $\overline{\mbR^{n+1}_+}$.
Then $\S$ is  an anisotropic free boundary hypersurface in $\mbR^{n+1}_+$ with respect to the Minkowski norm $\tilde F$, defined by
\eq{\label{defn-tilde-F}
\tilde F(\xi)
\coloneqq F(\xi)+\om_0\<\xi,E_{n+1}^F\>,\quad\forall\xi\in\mbR^{n+1}.
}
Moreover, 
\begin{itemize}
    \item $A_F=A_{\tilde F}$ on $\Sigma$, and hence the anisotropic curvatures of $\Sigma$ w.r.t. to $F$ and $\tilde F$ are the same.
    \item  The unit Wulff shapes of associated to $F$ and $\bar F$ are the same up to a translation, precisely, $$\p\mcW^{\tilde F}=\p\mcW_1^F(\om_0E_{n+1}^F).$$
\end{itemize}
\end{proposition}
\begin{proof}
By direct computation, we see that
\eq{\label{eq-DF-DtildeF}
D\tilde F(\xi)
=D F(\xi)+\om_0E_{n+1}^F.
}
It follows that for every $p\in\p\S\subset\p\mbR^{n+1}_+$, there holds
\eq{
\left<\nu_{\tilde F}(p),-E_{n+1}\right>
=\left<\nu_{F}(p)+\om_0E_{n+1}^F,-E_{n+1}\right>
=\om_0-\om_0=0,
}
where we have used the fact that $\left<E_{n+1},E^F_{n+1}\right>=1$.
This proves the first part of the assertion.

It is direct to see that $A_F(p)=A_{\tilde F}(p)$ for any $p\in\S$, since we have
\eq{
D^2\tilde F(\xi)
=D^2F(\xi).
}
The last part of the assertion follows from \eqref{eq-DF-DtildeF}, and the fact that the unit Wulff shape with origin as its center could be characterized by $\p\mcW^{\tilde F}=D\tilde F(\mbS^n)$.
\end{proof}

Since  $F$ is in general not an even function on $\mbS^n$, for later purpose, we need the following  Minkowski norm. Define $F_\ast$  by 
$$F_{\ast}(z)=F(-z), \quad z\in\mbS^n.$$
Geometrically speaking, $F_\ast$ is induced by the convex body 
{obtained by applying a central symmetric to $\mcW^F$.}
It is clear that $F_\ast$ also induces a smooth Minkowski norm, still denoted by $F_\ast$, and we 
denote by $\Phi_\ast$ and $F_\ast^o$ the Cahn-Hoffman map and the dual Minkowski norm associated to $F_\ast$. 
\begin{proposition}\label{Prop-F-F_ast}
There holds
\begin{enumerate}
    \item \eq{\label{eq-F^o-F^o_ast}
    F_\ast^o(x)&=F^o(-x),\quad\forall x\in \mbR^{n+1}.
    }
    \item
    \eq{\label{eq-Phi-Phi_ast}
\Phi_\ast(z)&=-\Phi(-z),\quad \forall z\in\mbS^n.
}
     \item  For any $x\in\S$,
     let $\{\kappa_i^F(x)\}_{i=1}^n$ denote the  anisotropic principal curvatures of $\S$ at $x$ with respect to $\nu$ and $F$.
     Then $\{-\kappa_i^F(x)\}_{i=1}^n$
     are the anisotropic principal curvatures of $\S$ at $x$ with respect to the unit inner normal $-\nu$ and $F_\ast$,  which we denote by $\{\kappa_i^{F_\ast}(x)\}_{i=1}^n$.
\end{enumerate}
\end{proposition}
\begin{proof}
To prove \eqref{eq-F^o-F^o_ast}, we verify by definition.
Precisely, thanks to \eqref{defn-dual}, we have
\eq{
F^o_\ast(x)
=\sup\left\{\frac{\left<x,z\right>}{F_\ast(z)}\Big| \, z\in\mbS^n\right\}
=&\sup\left\{\frac{\left<x,z\right>}{F(-z)}\Big| \, z\in\mbS^n\right\}\\
=&\sup\left\{\frac{\left<-x,-z\right>}{F(-z)}\Big| \, -z\in\mbS^n\right\}
=F^o(-x).
}
\eqref{eq-Phi-Phi_ast} follows directly by
differentiating both sides of the equality $F_\ast(z)=F(-z)$.

To prove (3), we fix any $x\in\S$ and let $\{e_i(x)\}_{i=1}^n$ denote the anisotropic principal vector at $x\in\S$ corresponding to $\kappa_i^F$, that is,
\eq{
D_{e_i(x)}\Phi(\nu(x))
=\kappa^F_i(x)e_i(x).
}
Letting $z=-\nu(x)$ in \eqref{eq-Phi-Phi_ast} then
differentiating with respect to $e_i(x)$, we obtain
\eq{
D_{e_i(x)}\Phi_\ast(-\nu(x))
=-D_{e_i(x)}\Phi(\nu(x))
=-\kappa^F_i(x)e_i(x),
}
thereby showing that $\{e_i(x)\}_{i=1}^n$ are also the anisotropic principal vectors at $x\in\S$, corresponding to $\{\kappa_i^{F_\ast}(x)\}_{i=1}^n$, with
\eq{
\kappa_i^{F_\ast}(x)
=-\kappa_i^F(x),\quad\forall i=1,\ldots,n.
}
This completes the proof.
\end{proof}

We close this subsection by stating the following geometric result, which follows
simply by using integration by parts for the divergence of the position vector field.
\begin{lemma}\label{Lem-area-volume}
We have
\eq{\label{eq-area-volume}
\int_{\p\mcW^F\cap \mbR^{n+1}_+}F(\nu) \rd A
=(n+1)\abs{\mcW^F\cap \mbR^{n+1}_+}}
and
\eq{\label{eq2.9}
\int_{\p\mcW^F_{1,\om_0}\cap\mbR^{n+1}_+}\left(F(\nu)+\om_0\left<\nu,E_{n+1}^F\right>\right)\rd A
=(n+1)\abs{\mcW_{1,\om_0}^F},
}
where $\mcW^F_{1,\om_0}=\mcW^F_1(\om_0 E^F_{n+1})\cap {\mbR_+^{n+1}}$.
\end{lemma}
\begin{proof}
     \eqref{eq-area-volume} follows from integration by parts of the divergence of the position vector field. \eqref{eq2.9}
follows from     plugging $\tilde F$ given by \eqref{defn-tilde-F} into \eqref{eq-area-volume}.
\end{proof}


\subsection{Properties of unbounded convex sets}\label{App-B}

The following two Propositions may be familiar to experts, especially in the case  $F(\xi)=\abs{\xi}$.

\begin{proposition}\label{Prop-avr-K}
Let $K$ be an unbounded, closed convex set with boundary $\p K$
which contains the origin.
Fix $0\leq r<+\infty$,
then 
$\frac{\vert \mcW^F_{r+R}\cap K\vert}{R^{n+1}}$ is non-increasing. In particular, $\frac{\vert\mcW^F_R\cap K\vert}{R^{n+1}}$ is a constant if and only if $K$ is a cone with vertex at the origin.
\end{proposition}

\begin{proof}[Proof of Proposition \ref{Prop-avr-K}]
By using the co-area formula, for any $0\leq r$, we have
\eq{
\frac{\rd}{\rd R}\frac{\vert \mcW^F_{r+R}\cap K\vert}{R^{n+1}}
=R^{-(n+1)}\left(\int_{\p\mcW^F_{r+R}\cap K}\frac{1}{|DF^o|}\rd A- (n+1)\frac{\vert \mcW^F_{r+R}\cap K\vert}{R}\right).
}
On the other hand, by the divergence theorem, we have
\eq{
(n+1)\vert\mcW^F_{r+R}\cap K\vert
=&\int_{\mcW^F_{r+R}\cap K}{\rm div}(x)\rd x\\
=&\int_{\p\mcW^F_{r+R}\cap K}\left\<x,\frac{DF^o(x)}{\abs{DF^o(x)}}\right\>\rd A
+\int_{\mcW^F_{r+R}\cap\p K}\<x,\bar N(x)\>\rd A\\
\geq&\int_{\p\mcW^F_{r+R}\cap K}\frac{R}{\abs{DF^o}}\rd A.
}
In the last inequality, we have used $\<x, \bar N(x)\>\ge 0$ thanks to the convexity of $K$.

It is then direct to see that
\begin{eqnarray*}
\frac{\rd}{\rd R}\frac{\vert\mcW^F_{r+R}\cap K\vert}{R^{n+1}}\le 0.
\end{eqnarray*}

If
$\frac{\rd}{\rd R}\frac{\vert\mcW^F_R\cap K\vert}{R^{n+1}}= 0,$ then
$\<x,\bar N(x)\>=0$ for $\mcH^n$-a.e. $x\in\p K$, it follows that $K$ is a cone with vertex at the origin, see e.g., \cite{Mag12}*{Proposition 28.8}.
\end{proof}

A direct consequence is that one can define the \textit{asymptotic volume ratio with respect to $F$} for $K$ as
\eq{\label{defn-AVR(K)}
 {\rm AVR}_F(K)
 =\lim_{R\to\infty} \frac{\vert \mcW^F_R\cap K\vert}{\abs{\mcW^F}R^{n+1}}.
}

Similarly, we show the following asymptotic volume ratio for any $F$.

\begin{proposition}\label{Prop-tangent-cone-infinity}
Let $K$ be an unbounded, closed convex set with boundary $\p K$
which contains the origin.
There is a unique \textit{tangent cone at infinity} of $K$, say $K_\infty$.
Moreover, one has
\eq{
{\rm AVR}_F(K)
=\frac{\vert\mcW^F \cap K_\infty\vert}{\abs{\mcW^F}}.
}
\end{proposition}

\begin{proof}[Proof of Proposition \ref{Prop-tangent-cone-infinity}]
Denote by $K_R=\frac{1}{R}K$. From the compactness result, see e.g., \cite{Mag12}*{Corollary 12.27}, we know that, up to a subsequence, say $\{R_h\}_{h\in\mbN}$,
\eq{
K_{R_h}\xrightarrow{{\rm loc}}K_\infty
}
for some sets of locally finite perimeter $\mcC$, with convergence in the sense that, for every compact set $E\subset\mbR^{n+1}$,
\eq{\label{eq-limit-Difference-Set}
\lim_{h\ra\infty}\Abs{E\cap\left(K_\infty\De K_{R_h}\right)}=0.
}
By the monotonicity, one can prove that $K_\infty$ must be a cone. In fact,
for each $s>0$, we have
\eq{
\lim_{h\to\infty}\vert \mcW^F_s\cap K_{R_h}\vert
=\lim_{h\to\infty}\frac{\vert \mcW^F_{sR_h}\cap K\vert}{R_h^{n+1}}
\overset{\eqref{defn-AVR(K)}}{=}
s^{n+1}{\rm AVR}_F(K)\abs{\mcW^F}.
}
On the other hand, from \eqref{eq-limit-Difference-Set} we deduce
\eq{
\lim_{h\to\infty}\vert \mcW^F_s\cap K_{R_h}\vert
=\vert \mcW^F_s\cap K_\infty\vert.
}
It follows that, regardless of the choice of convergence subsequence, there always holds
\eq{
\frac{\vert \mcW^F_s\cap K_\infty\vert}{s^{n+1}}
={\rm AVR}_F(K)\abs{\mcW^F}.
}
As the LHS is a constant about $s$, it follows from Proposition \ref{Prop-avr-K} that
$\mcC$ is a cone,
revealing the fact that
\eq{
{\rm AVR}_F(K)
=\frac{\vert\mcW^F\cap K_\infty\vert}{\abs{\mcW^F}}.
}
In other words, we have proved that for any such non-compact convex set $K$, there exists a unique \textit{tangent cone at infinity}, denoted by $K_\infty$.
This completes the proof.
\end{proof}

\section{Willmore inequality in unbounded convex sets}\label{Sec-3}

As in \cite{JWXZ23}, we introduce the following flow which generates \textit{parallel hypersurfaces} in the anisotropic free boundary sense, defined by
\eq{\label{defn-zeta_om_0}
\zeta_{F}(x,t)
=x+t\Phi(\nu(x)),\text{ }(x,t)\in\S\times\mbR_+ \eqqcolon Z.
}
In the rest of the section,
we write $\mcW^{F}_{r,K}=\mcW^{F}_{r}(0)\cap{K}$ for any $r>0$.

\begin{proposition}\label{Prop-surjective-halfspace-ani}
Given an unbounded, closed convex sets $K$  in $\mbR^{n+1}$.
Let $\S\subset K$ be a compact, embedded $C^2$-hypersurface with boundary $\p\S\subset{\rm Reg}(\p K)$ such that \eqref{condi-convexcone-ani} holds.
Denote by $\Om$ the bounded domain delimited by $\S$ and $\p K$.
Assume that $0\in {\rm int}(\p\Om\cap\p K)$.
Define a distance function on $K$ by 
\eq{\label{defn-d-om_0}
d_{F_\ast}(y,\overline\Omega)=\inf_{r>0}\left\{r\mid \overline{\mcW^{F_\ast}_r(y)}\cap\overline{\Omega}\not=\emptyset\right\}.
}
Then we have the following statements:
\begin{enumerate}
    \item For any $R>0$, there holds
\eq{\label{ineq-set-aniso-halfspace}
&\left\{y\in {K}\setminus \overline{\Omega}:d_{F_\ast}(y,\overline\Omega) \leq R\right\}\\
\subset&\left\{\zeta_F(x,t)\mid x\in \Sigma,t\in\left(0,\min(R,\tau(x))\right]\right\},
}
where
$\tau$ is a function defined on $\S$ by 
\eq{\label{defn-tau-halfspace-ani}
\tau(x)=
\begin{cases}
+\infty, &\text{if }\kappa_i^{F_\ast}(x)\leq 0 \text{ for any } i=1,\cdots,n,\\
\frac{1}{\max_i \kappa_i^{F_\ast}(x)},&\text{otherwise}.
\end{cases}
}
    \item 
    For any $R>0$, there holds
\eq{\label{eq-d_om-r+R}
\left\{  y\in {K}:d_{F_\ast}\left(  y,\overline{\mcW^{F}_{r,K}}\right) \leq R\right\}
=\overline{\mcW^{F}_{r+R,K}}.
}
    \item 
\eq{\label{eq-avr-halfspace-ani}
 \lim_{R\rightarrow +\infty}\frac{{\rm Vol}\left(\left\{  y\in {K}:d_{F_\ast}\left(  y,\overline{\Om}\right) \leq R\right\}\right)}{\abs{\mcW^F}R^{n+1}}
 ={\rm AVR}_F(K),
 }
 where ${\rm AVR}_F(K)$ is defined as \eqref{defn-AVR(K)}.
\end{enumerate}
\end{proposition}
\begin{proof}
(1):
For any $y\in {K}\setminus \overline{\Omega}$ satisfying $d_{F_\ast}(y,\overline\Omega)\leq R$, we use a family of closed Wulff balls $\{\overline{\mcW^{F_\ast}_r(y)}\}_{r>0}$ to touch $\overline\Omega$.
Clearly there must exist $x\in\partial \Omega$ and $0<r_0\leq R$,
such that  $\overline{\mcW_{r_0}^{F_\ast}(y)}$ touches $\overline\Omega$ for the first time at a point $x$.
Let $\nu^\ast(x)$ be  the outer unit normal of $\overline{\mcW_{r_0}^{F_\ast}(y)}$ at $x$, and $\Phi_\ast(\nu^\ast(x))$ the anisotropic normal at $x$ satisfying 
\eq{
\Phi_\ast(\nu^\ast(x))
=\frac{x-y}{r_0}.
}
By convexity of $K$ and strictly convexity of $\mcW_{r_0}^{F_\ast}(y)$,  $x$ cannot be obtained at $\partial\Omega\setminus\Sigma$, thus only the following two cases are possible:

\textbf{Case 1.}  $x\in \partial\Sigma$.

Since $\overline{\mcW^{F_\ast}_{r_0}(y)}$ touches $\overline\Omega$ from outside at $x$, $\nu^\ast(x)$, $-\nu(x)$ and $-\bar N(x)$ lie on the same $2$-plane and moreover, $\nu^\ast$ lies in a length-minimizing geodesic on $\mbS^1$ joining $-\nu$ and $-\bar N(x)$.
Thanks to Proposition \ref{Prop-angle-compare}, we find
\eq{
\<\Phi_\ast(\nu^\ast(x)),-\bar N(x)\>
&\geq \<\Phi_\ast(-\nu(x)),-\bar N(x)\>
\\&=\<-\Phi(\nu(x)),-\bar N(x)\>
=\<\nu_F(x),\bar N(x)\>
\geq0.
}
On the other hand, since $K$ is convex and $y\in K$, we must have
\eq{
0
\geq\<y-x,\bar N(x)\>=\<-r_0\Phi_\ast(\nu^\ast(x)), \bar N(x)\>
= r_0\<\Phi_\ast(\nu^\ast(x)),-\bar N(x)\>.
}
 Hence, the above two inequalities must hold as equalities simultaneously; that is to say, $x$ must belong to $\{x\in\partial\Sigma\mid \<\nu_F(x),\bar N(x)\>=0\}$, while $\nu^\ast(x)=-\nu(x)$, then there hold that
\eq{
y=x-r_0\Phi_\ast(-\nu) 
=x+r_0\Phi(\nu)
=\zeta_{F}(x,r_0),
}
and that $\overline\Omega$ and   $\overline{\mcW^{F_\ast}_{r_0}(y)}$ are mutually tangent at x.
 Recall that $\kappa_i^{F_\ast}(x)$ denotes the anisotropic principal curvatures of $\S$ at $x$, with respect to $-\nu$ and $F_{\ast}$, we therefore find $\max_i\kappa_i^{F_\ast}(x)\leq\frac1{r_0}$.
Taking the definition of $\tau(x)$ into account, this readily implies $r_0\leq\tau(x)$.

\textbf{Case 2.} $x\in {\rm int}(\Sigma)$.

In this case, it is easy to see that $\nu^\ast(x)=-\nu(x)$ from the first touching property. We may conduct a similar 
argument as above to find that $y$ can be written as $y=\zeta_F(x,r_0)$, with $r_0\leq\tau(x)$.

Therefore, for any $y\in {K}\setminus \overline{\Omega}$, if  $d_{F_\ast}(y,\overline\Omega)\leq R$, we can find $x\in \Sigma$ and $r_0\in\left(0,\min(R,\tau(x))\right]$, such that $y=\zeta_F(x,r_0)$, which implies  \eqref{ineq-set-aniso-halfspace}.

\

\noindent(2): Now we show \eqref{eq-d_om-r+R}.
Recall that we have set $\mcW^{F}_{r,K}=\mcW^{F}_{r}(0)\cap K$.

Our first observation is that, 
if $y\in\overline{\mcW^{F}_{r,K}}$, it is easy to see that $d_{F_\ast}\left( y,\overline{\mcW^{F}_{r,{K}}}\right)=0 $.
Since it is trivial to see that $\mcW^{F}_{r,K}\subset \mcW^{F}_{r+R,K}$, to prove \eqref{eq-d_om-r+R}, it suffices to show that
\eq{\label{eq-W_{r+R}-W_r}
  \left\{  y\in {K}\setminus\overline{\mcW^{F}_{r,{K}}}:d_{F_\ast}\left(y,\overline{\mcW^{F}_{r,{K}}}\right) \leq R\right\}=\overline{\mcW^{F}_{r+R,{K}}}\setminus\overline{\mcW^{F}_{r,{K}}}, \quad \forall R>0.
}

{\bf Claim.}
For any $ y\in {K}\setminus \overline{\mcW^F_{r,K}}$,
let $\tilde R=d_{F_\ast}\left(y,\overline{\mcW^F_{r,{K}}}\right)$,
then $F^o(y)=r+\tilde R$.

To prove the claim,
we consider the family of closed Wulff balls $\{\overline{\mcW^{F_\ast}_r(y)}\}_{r>0}$, as $r$ increases, by the definition of $d_{F_\ast}$, we have that
$\overline{\mcW_{\tilde R}^{F_\ast}(y)}$ touches $\overline{\mcW^{F}_{r,K}}$ for the first time. 
 
Suppose that  $\overline{\mcW_{\tilde R}^{F_\ast}(y)}$ touches $\overline{\mcW^{F}_{r,K}}$ at $x$, then
\eq{x\in \partial \mcW^{F_\ast}_{\tilde R}(y)\cap\partial \mcW^{F}_{r,K}.
}
Since
$\overline{\mcW_{\tilde R}^{F_\ast}(y)}$ is strictly convex, 
and $K$ is convex,
we deduce that 
\eq{
x\in \partial \mcW^{F_\ast}_{\tilde R}(y)\cap\partial\mcW^{F}_{r}(0)\cap {K}.
}
Let $\nu(x)$, $\tilde\nu(x)$ denote respectively the unit outer normals of $\mcW^{F_\ast}_{\tilde R}(y)$ and $\mcW^{F}_{r}(0)$ at $x$.
It follows that
  \eq{\label{eq-normal-wulffshape-halfspace}
 x-y=\tilde R \Phi_\ast(\nu(x)),\quad
 x=r\Phi(\tilde \nu(x)).
 }
 
If $x\in \partial \mcW^{F_\ast}_{\tilde R}(y)\cap\partial\mcW^{F}_{r}(0)\cap \partial K$, we have $-\nu(x)$ lies in a length-minimizing geodesic joining $\tilde\nu(x)$ and $\bar N(x)$ in $\mbS^n$.
 By Proposition \ref{Prop-angle-compare}, it holds that
\eq{
\<\Phi(\tilde\nu(x)),\bar N(x)\>\leq\<\Phi(-\nu(x)), \bar N(x)\>=-\<\Phi_\ast(\nu(x)),\bar N(x)\>.
}
On the other hand, it follows from \eqref{eq-normal-wulffshape-halfspace}  and convexity of $K$ that
\eq{
\<\Phi(\tilde\nu(x)),\bar N(x)\>
=\left\<\frac{x}{r},\bar N(x)\right\>\geq0,\\
-\<\Phi_\ast(\nu(x)),\bar N(x)\>
=-\left\<\frac{x-y}{\tilde R},\bar N(x)\right\>
\leq0.
}
Combining all above,
we thus obtain $\nu(x)=-\tilde\nu(x)$ and
$\<x,\bar N(x)\>=\left<y,\bar N(x)\right>=0$.

If $x\in \partial \mcW^{F_\ast}_{\tilde R}(y)\cap\partial\mcW^{F}_{r}(0)\cap{\rm int}(K)$, it is easy to see that  $\nu(x)=-\tilde\nu(x)$.
 
Summarizing, in either case, we always have $\nu(x)=-\tilde\nu(x)$, thus
\eq{
\Phi_\ast(\nu)=-\Phi(-\nu)=-\Phi(\tilde\nu).
}
Using \eqref{eq-normal-wulffshape-halfspace} again, we get 
\eq{
F^o(y)
=&F^o(y-x +x)\\
=&F^o(-\tilde R\Phi_\ast(\nu(x))+r\Phi(\tilde\nu(x)))
=F^o((r+\tilde R)\Phi(\tilde\nu(x)))
=r+\tilde R,
}
which proves the claim.
Now we prove \eqref{eq-W_{r+R}-W_r}.
 
{\bf"$\subset$":} For any $ y\in  K\setminus \overline{\mcW^F_{r,K}}$ satisfying $\tilde R=d_{F_\ast}\left(  y,\overline{\mcW^{F}_{r,K}}\right)\leq R $, we deduce immediately from the above estimate that
\eq{
F^o\left(y\right)
= r+\tilde R
\leq r+R,
}
thus $y\in \overline{\mcW^F_{r+R,K}}\setminus \overline{\mcW^F_{r,{K}}}$.
 
{\bf "$\supset$":} For any $y\in \overline{\mcW^F_{r+R,{K}}}\setminus \overline{\mcW^F_{r,{K}}}$, suppose that $\tilde R=d_{F_\ast}\left(y,\overline{\mcW^F_{r,K}}\right)> R $, then it holds that
\eq{
F^o\left(y\right)
=r+\tilde R
>r+R,
}
which contradicts to $y\in \overline{\mcW^F_{r+R,K}}$. Hence $d_{F_\ast}\left(y,\overline{\mcW^F_{r,{K}}}\right)\leq R$, which proves (2).

\

\noindent(3):
Since $\Omega$ is bounded, we could find $r_1$ and $r_2$ such that $\mcW^F_{r_1,{K}}\subset \Omega\subset \mcW^F_{r_2,{K}}$. From the definition of $d_{F_\ast}$ \eqref{defn-d-om_0}, we find 
\eq{
d_{F_\ast}(y,\overline{\mcW^F_{r_2,{K}}})\leq d_{F_\ast}(y,\overline{\Om})\leq d_{F_\ast}(y, \overline{\mcW^F_{r_1,{K}}}),
\quad \forall y\in{K},
}
and it follows that
\eq{\label{inclu-sets-halfspace-ani}
& \left\{  y\in {K}:d_{F_\ast}\left(  y,\overline{\mcW^F_{r_1,{K}}}\right)\leq R\right\}\\
\subset&\left\{  y\in {K}:d_{F_\ast}\left(  y,\overline\Omega\right) \leq R\right\}\\
 \subset &\left\{  y\in {K}:d_{F_\ast}\left(  y,\overline{\mcW^F_{r_2,{K}}}\right) \leq R\right\}.
}
Dividing \eqref{inclu-sets-halfspace-ani} by $\abs{\mcW^F}R^{n+1}$, we get \eqref{eq-avr-halfspace-ani}  by using \eqref{eq-d_om-r+R} and Proposition \ref{Prop-avr-K}.
\end{proof}

We have now all the prerequisites to prove the Willmore inequality.

\begin{proof}[Proof of Theorem \ref{Thm-aniso-convex}]
We first prove the Willmore inequality.
Our starting point is that, thanks to \eqref{ineq-set-aniso-halfspace}, we may use the area formula
{(see e.g., \cite{Simon83}*{8.5})}
to estimate the volume as follows: for any $R>0$,
\eq{\label{ineq-willmore-aniso-halfspace-1}
{\rm Vol}\left(\left\{  y\in {K}:d_{F_\ast}\left(  y,\overline\Omega\right)
\leq R\right\}\right)
\leq \left\vert\Omega\right\vert
+\int_{\Sigma}\int_{0}^{\min\left(  R,\tau\left(  x\right)\right)}{\rm J}^Z{\zeta_F}(x,t)\rd t\rd A.
}
By a simple computation, we see, the tangential Jacobian of $\zeta_F$ along $Z=\S\times\mbR_+$ at $(x,t)$ is just
$${\rm J}^Z{\zeta_F}(x,t)
= F(\nu(x))\prod_{i=1}^n(1+t\kappa_i^F(x)).
$$
Recalling the definition of $\tau$ in \eqref{defn-tau-halfspace-ani}, and also taking Proposition \ref{Prop-F-F_ast}(3) into account, 
we may rewrite $\tau$ as
\eq{
\tau(x)=
\begin{cases}
+\infty, &\text{if }\kappa_i^F(x)\geq 0 \text{ for any } i=1,\cdots,n,\\
-\frac{1}{\min_i \kappa_i^F(x)},&\text{otherwise}.
\end{cases}
}
It is clear that for any $x\in\S$, there holds $1+t\kappa_i^F(x)>0$, for each $i=1,\cdots,n$, and for any $t\in(0,\tau(x))$.
By the
AM-GM inequality, 
we have 
\eq{\label{ineq-AM-GM-aniso-halfspace}
{\rm J}^Z{\zeta_F}(x,t)
\leq F(\nu)\left(1+{H^F(x)}t\right)^n.
}

To have a closer look at \eqref{ineq-willmore-aniso-halfspace-1}, we divide $\S$ into two parts $\S_+=\{x\in\Sigma:H^F(x)>0\}$ and $\S\setminus\Sigma_+$. 
On $\S\setminus\S_+$, we have 
\eq{
0\leq \left(1+{H^F}t\right)^n\leq 1, \forall t\in [0,\tau(x)],
}
which, in conjunction with \eqref{ineq-AM-GM-aniso-halfspace}, gives
\eq{\label{ineq-partofsigma-aniso-halfspace}
&\int_{\S\setminus\S_+}\int_{0}^{\min\left(R,\tau\left(x\right)\right)}{\rm J}^Z{\zeta_F}(x,t)\rd t\rd A
\le O(R).
}
Thus \eqref{ineq-willmore-aniso-halfspace-1} can be further estimated as follows:
\eq{\label{ineq-willmore-aniso-halfspace-2}
&{\rm Vol}\left(\left\{y\in K:d_{F_\ast}\left(  y,\overline\Om\right)\leq R\right\}\right)\\
\leq&\abs{\Om}
+\int_{\S_+}\int_{0}^{\min\left(  R,\tau\left(x\right)\right)}{\rm J}^Z{\zeta_F}\rd t\rd A 
+\int_{\S\setminus\S_+}\int_{0}^{\min\left(R,\tau\left(x\right)\right)}{\rm J}^Z{\zeta_F}\rd t\rd A\\
\leq&\int_{\S_+}\int_{0}^{R}F(\nu)\left(1+{H^F(x)}t\right)^n\rd t\rd A+O(R)\\
=&\frac{R^{n+1}}{n+1}\int_{\S_+}F(\nu)\left({H^F(x)}\right)^n\rd A
+O(R^{n})\\
\leq&\frac{R^{n+1}}{n+1}\int_{\S}F(\nu)\Abs{{H^F}}^{n}\rd A
+O(R^{n}),
}
where the third inequality holds due to \eqref{ineq-AM-GM-aniso-halfspace}.

Dividing  both sides of \eqref{ineq-willmore-aniso-halfspace-2} by $R^{n+1}$ and
letting $R\rightarrow +\infty$, we deduce from \eqref{eq-avr-halfspace-ani} that 
\eq{\label{xeq-1}
{\rm AVR}_F(K)\abs{\mcW^F}
\le \frac{1}{n+1}\int_{\S_+}F(\nu){\left({H^F}\right)}^{n}\rd A \leq\frac{1}{n+1}\int_{\S}F(\nu)\abs{{H^F}}^{n}\rd A.
}
which finishes the proof of the inequality \eqref{ineq-willmore-aniso-cone}.


Now we show the rigidity. First we observe:

{\bf Claim 1. }
If the equality in \eqref{ineq-willmore-aniso-cone} holds, 
then $H^F\geq0$ on $\S$,
 $\tau(x)=+\infty$ on $\S_+$, and equality in \eqref{ineq-AM-GM-aniso-halfspace} holds  on $\S_+$.

To prove the first assertion of the claim,
we deduce from \eqref{xeq-1} and the equality case of \eqref{ineq-willmore-aniso-cone} that
\eq{\label{eq:S-S_+}
\frac{1}{n+1}\int_{\S}F(\nu)\abs{{H^F}}^{n}\rd A
= \frac{1}{n+1}\int_{\S_+}F(\nu){\left({H^F}\right)}^{n}\rd A,
}
which implies $H^F\geq 0$ on $\S$.

To prove $\tau(x)=+\infty$ on $\S_+$, 
we argue by contradiction and
assume that there exists a point $x_0\in \S_+$ satisfying $\tau(x_0)<+\infty$. Since $\S\in C^2$, we can find a neighborhood of $x_0$ in $\S_+$, denoted by $U(x_0)$, such that $\tau(x)\leq 2\tau(x_0)<+\infty$ on $U(x_0)$.  For any $R>2\tau(x_0)$, we obtain from \eqref{ineq-willmore-aniso-halfspace-2} that
\eq{
&{\rm Vol}\left(\left\{y\in K:d_{F_\ast}\left(  y,\overline\Om\right)\leq R\right\}\right)\\
\leq& \int_{\S_+}\int_{0}^{\min{(R,\tau(x))}}F(\nu)\left(1+{H^F(x)}t\right)^n\rd t\rd A +O(R)\\
\leq& \int_{\S_+\setminus U(x_0)}\int_{0}^{R}F(\nu)\left(1+{H^F(x)}t\right)^n\rd t\rd A \\
&+\int_{ U(x_0)}\int_{0}^{\tau(x)}F(\nu)\left(1+{H^F(x)}t\right)^n\rd t\rd A +O(R)\\
\leq&\int_{\S_+\setminus U(x_0)}\int_{0}^{R}F(\nu)\left(1+{H^F(x)}t\right)^n\rd t\rd A  +O(R)\\
=&\frac{R^{n+1}}{n+1}\int_{\Sigma_+\setminus U(x_0)} F(\nu)\left({H^F}\right)^n\rd A +O(R^n).
}
Dividing  both sides by $R^{n+1}$ and
letting $R\rightarrow +\infty$, it follows that 
\eq{
{\rm AVR}_F(K)\abs{\mcW^F}
\leq\frac{1}{n+1}\int_{\S_+\setminus U(x_0)}F(\nu){\left({H^F}\right)}^{n}\rd A,
}
which is a contradiction to the assumption that equality in \eqref{ineq-willmore-aniso-cone} holds. 

By a similar contradiction argument, we also get that
equality in \eqref{ineq-AM-GM-aniso-halfspace} holds  on $\S_+$, and we omit the proof here for brevity. 
In particular, this proves {\bf Claim 1}.

 
By virtue of this claim, we see that  $\S=\S_+$ must be an anisotropic umbilical hypersurface, and hence a part of a Wulff shape.
Write $\S={\p\mcW_{R_0}^F(x_0)\cap K}$ for some $R_0>0$ and $x_0\in \mbR^{n+1}$.
Clearly, $\S$ determines an solid cone with vertex at $x_0$, given by
\eq{
\mcC_{x_0}
=\left\{x_0+t\frac{x-x_0}{R_0}:x\in {\rm int}(\S), t\in(0,+\infty)\right\}.
}
Consider now the modified convex set $$\widetilde K=\left(K\setminus\Om\right)\cup\left(\mcC_{x_0}\cap\mcW_{R_0}^F(x_0)\right).$$
{Due to the convexity of $K$ and the boundary condition \eqref{condi-convexcone-ani}, for any $x\in\partial\Sigma$, there exists at least one supporting hyperplane passing through $x$, hence $\widetilde K$ is also convex.}

{\bf Claim 2.} $\frac{\abs{\mcW^F_R(x_0)\cap\widetilde K}}{R^{n+1}}$ is non-increasing on $R\in [0,+\infty)$.

This can be proved similarly as Proposition \ref{Prop-avr-K}.
Indeed,
by using the co-area formula, we have 
\eq{
\frac{\rd}{\rd R}\frac{\vert \mcW^F_{R}(x_0)\cap \widetilde  K\vert}{R^{n+1}}
=\frac{\int_{\p\mcW^F_{R}(x_0)\cap \widetilde K}\frac{1}{|DF^o(x-x_0)|} \rd A- (n+1)\frac{\vert \mcW^F_{R}(x_0)\cap \widetilde  K\vert}{R}}{R^{n+1}}.
}
On the other hand, by the divergence theorem, we have for $R\ge R_0$,
\eq{
(n+1)\vert\mcW^F_{R}(x_0)\cap \widetilde  K\vert
=&\int_{\mcW^F_{R}(x_0)\cap\widetilde K}{\rm div}(x-x_0)\rd x \\
=&\int_{\p\mcW^F_{R}(x_0)\cap\widetilde K}\frac{R}{\abs{DF^o(x-x_0)}}\rd A\\
&+\int_{\mcW^F_{R}(x_0)\cap\p K\setminus \Omega}\<x-x_0,\bar N(x)\>\rd A
+\int_{\p \mcC_{x_0}\cap\mcW^F_{R_0}(x_0)}\<x-x_0,\bar N(x)\>\rd A.
}
Note that $\<x-x_0,\bar N(x)\>=0$ on $\p \mcC_{x_0}\cap\mcW^F_{R_0}(x_0)$ since it is part of the boundary of a cone, and $\<x-x_0,\bar N(x)\>\ge0$ on $\mcW^F_{R}(x_0)\cap\p K\setminus \Omega$ since ${\widetilde{K}}$ is convex. 
It follows that  
\eq{
(n+1)\vert\mcW^F_{R}(x_0)\cap \widetilde  K\vert
\ge &R\int_{\p\mcW^F_{R}(x_0)\cap\widetilde K}\frac{1}{\abs{DF^o(x-x_0)}}\rd A, 
}
and in turn $$\frac{\rd}{\rd R}\frac{\vert \mcW^F_{R}(x_0)\cap \widetilde K\vert}{R^{n+1}}\le 0,$$
which proves the claim.

From the second claim, we know that 
\eq{
{\rm AVR}_F(K)\abs{\mcW^F}
=\lim_{R\to+\infty}\frac{\vert \mcW^F_{R}(x_0)\cap \widetilde  K\vert}{R^{n+1}}
\le\frac{\vert \mcW^F_{R_0}(x_0)\cap \widetilde  K\vert}{R_0^{n+1}}
=\frac{\vert \mcW^F_{R_0}(x_0)\cap \mcC_{x_0}\vert}{R_0^{n+1}}.
}
On the other hand, 
using  equality in \eqref{ineq-willmore-aniso-cone} and the fact that $\S$ is anisotropic umbilical, we get
\eq{
{\rm AVR}_F(K)\abs{\mcW^F}
=&\frac{1}{n+1}\frac{1}{R_0^n}\int_{\p\mcW_{R_0}^F(x_0)\cap K}F(\nu) \rd A
=\frac{1}{n+1}\frac{1}{R_0^n}\int_{\p\mcW_{R_0}^F(x_0)\cap \mcC_{x_0}}F(\nu) \rd A\\
=&\frac{\vert \mcW^F_{R_0}(x_0)\cap\mcC_{x_0}\vert}{R_0^{n+1}},
}
which 
implies that $\frac{\vert \mcW^F_{R}(x_0)\cap \widetilde K\vert}{R^{n+1}}$ is a constant for $R\in[R_0,+\infty)$.

\begin{figure}[H]
	\centering
	\includegraphics[width=12cm]{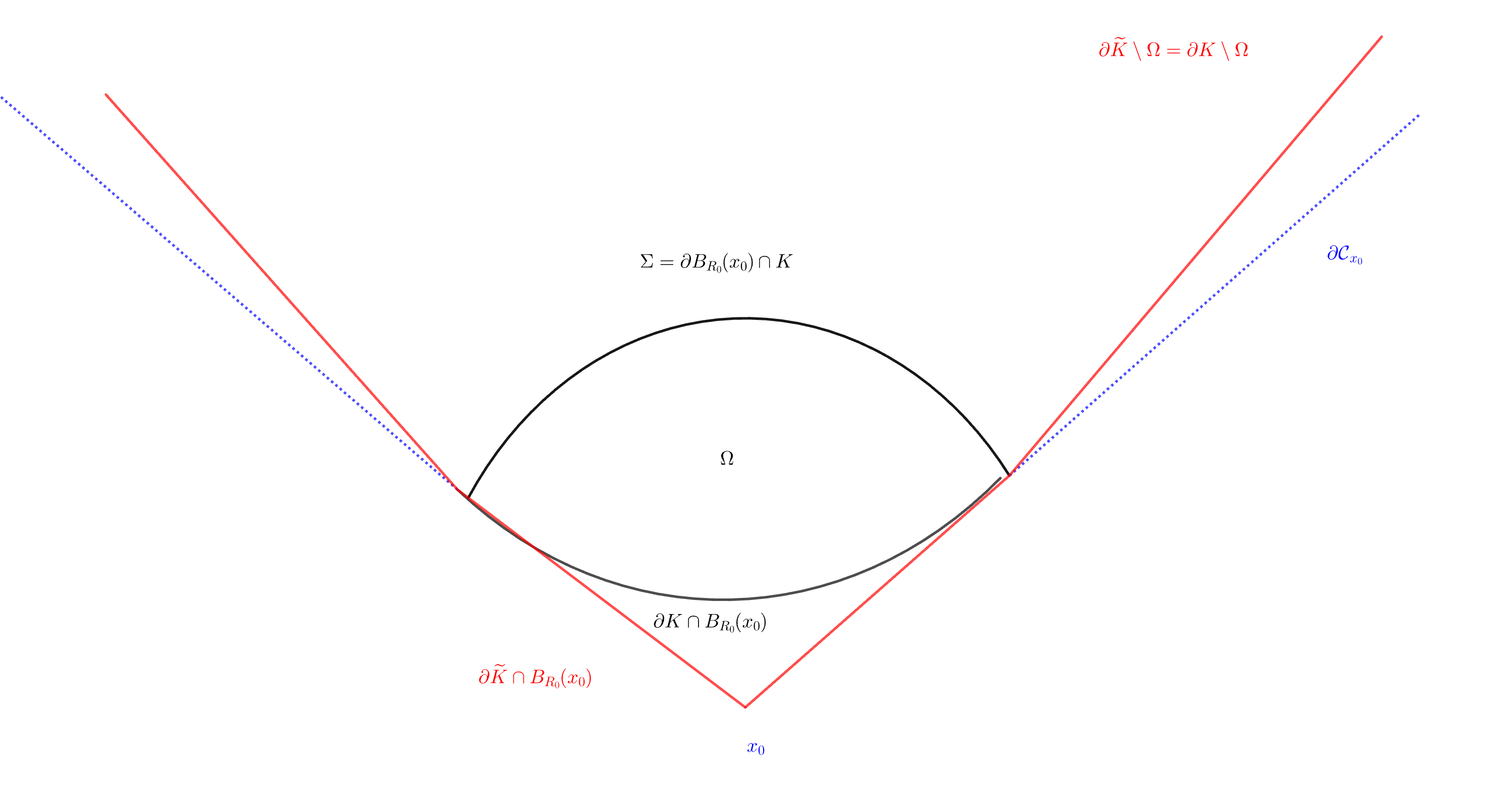}
	\caption{$\widetilde K$ and $\mathcal{C}_{x_0}$.}
	\label{Fig-1}
\end{figure}

See Fig. \ref{Fig-1} for an illustration when $F$ is the Euclidean norm, where the red part denotes the boundary of the modified set $\widetilde K$ and the blue part denotes the boundary of the cone $\mcC_{x_0}$.
Because of the boundary condition \eqref{condi-convexcone-ani}, we point out that Fig. \ref{Fig-1} is what we could expect so far.
Next, we are going to show by virtue of the claim that $\widetilde K$ is in fact $\mcC_{x_0}$, namely, the blue portion and the red portion in Fig. \ref{Fig-1} actually coincide.
In fact,
from the proof of the claim, we see that
$$\<x-x_0,\bar N(x)\>=0 \hbox{ on }\p K\setminus \Omega,$$
which implies that $\p K\setminus \Omega$ is a part of boundary of the cone centered at $x_0$. Finally from the information that
$\S$ is a part of Wulff shape centered at $x_0$ and $\p K\setminus \Omega$ is part of boundary of the cone centered at $x_0$, we see that $\S$ is indeed an anisotropic free boundary Wulff shape in $K$, i.e., $\<\nu_F,\bar N\>=0$ along $\p\S$.
This completes the proof.
\end{proof}

\section{Willmore inequalities in a half-space}\label{Sec-4}

\subsection{An alternative approach in the half-space case}In this section, we use an alternative approach to prove Theorem \ref{Thm-aniso-convex} in the half-space case. We restate it here.
\begin{theorem}\label{Thm-aniso-halfspace}
Let $\S\subset\overline{\mbR^{n+1}_+}$ be a compact, embedded, $C^2$-hypersurface with boundary $\p\S\subset\p\mbR_+^{n+1}$ such that
\eq{
  \<\nu_F(x), -E_{n+1}\>\geq0,\quad  \hbox{ for  any }x\in \partial\Sigma.
}
Then we have
\eq{\label{ineq-willmore-aniso-halfspace}
\frac1{n+1}\int_{\S} F(\nu)\abs{{H^F}}^n \rd A\ge  \vert \mcW^F\cap {\mbR_+^{n+1}}\vert.
}
Equality in \eqref{ineq-willmore-aniso-halfspace} holds if and only if $\S$ is an anisotropic free boundary Wulff shape.
\end{theorem}
The alternative approach is based on  the following geometric observation for the Gauss image of $\S$.
\begin{proposition}\label{Prop-Gauss-map}
Let $\S\subset\overline{\mbR^{n+1}_+}$ be a compact, embedded, $C^2$-hypersurface with boundary $\p\S\subset\p\mbR_+^{n+1}$ such that
\eq{
  \<\nu_F(x), -E_{n+1}\>\geq0,\quad  \hbox{ for  any }x\in \partial\Sigma.
}
Then the anisotropic Gauss map $\nu_F:\S\ra\p\mcW^F$ satisfies
\eq{
(\p\mcW^F\cap {\mbR_+^{n+1}})
\subset \nu_F(\widetilde{\S_+}).
}
where $\widetilde{\S_+}$ is the subset of $\S$ where the anisotropic Weingarten map $\rd\nu_F$ is positive-semi definite.
\end{proposition}

\begin{proof}
For any $y\in\p\mcW^F\cap \mbR^{n+1}_+$, denote by $N(y)$ the outer unit normal to $\p\mcW^F$ at $y$. Then $\Phi(N(y))=y$.
Let ${\mfH}_y$ be a closed half-space whose inward unit normal is given by $N(y)$, such that $\S\cap \mfH_y=\emptyset$, and let ${\bf\Pi}_y$ be the boundary of $\mfH_y$, which is a hyperplane in $\mbR^{n+1}$.
Parallel translating ${\bf\Pi}_y$ towards $\S$ and denote by $\Pi_y$ the hyperplane that touches $\S$ for the first time at some $p\in\S$.
If $p\in \p\S\subset\p\mbR^{n+1}_+$, then by the first touching property, we have
\eq{\<\nu(p), -E_{n+1}\>\le \<N(y), -E_{n+1}\>.
}
Thanks to the first touching property, $\nu(p),N(y),$ and $-E_{n+1}$ are on the same $2$-plane and from the above relation we know, $N(y)$ lies in the length-minimizing geodesic on $\mbS^n$ joining $\nu(p)$ and $-E_{n+1}$.
Therefore using the anisotropic angle comparison principle (Proposition \ref{Prop-angle-compare}), we obtain
\eq{
\left<\nu_F(p),-E_{n+1}\right>
\leq\left<\Phi(N(y)),-E_{n+1}\right>= \left<y,-E_{n+1}\right>.
}
On the other hand, by assumption,
\eq{\<\nu_F(p), -E_{n+1}\>
\geq0
>\<y, -E_{n+1}\>,
}
which gives a contradiction.
Hence the first touching point $p\in{\rm int}(\S)$ and $\S$ is again tangent to $\Pi_y$ at $p$. Thus
$\nu(p)=N(y)$ and $\nu_F(p)=\Phi(\nu(p))=y$.
Note that the touching of $\S$ with $\Pi_y$ is from the interior, therefore at the touching point $p$, $\rd\nu(p)\ge 0$ and in turn, $\rd\nu_F(p)\ge 0$.
Thus we have $p\in \widetilde{\S_+}$ and $\nu_F(p)=y$, which completes the proof.
\end{proof}

\begin{proof}[Proof of Theorem \ref{Thm-aniso-halfspace}]
Our starting point is Proposition \ref{Prop-Gauss-map}.
To proceed, note that the Jacobian of the Gauss map with respect to $\S$ is just ${\rm J}^\S\rd\nu_F= H_n^F$,  where $H^F_n$ is the anisotropic Gauss-Kronecker curvature. By using \eqref{eq-area-volume}, the area formula and AM-GM inequality, we have
\eq{
(n+1)\abs{\mcW^F\cap \mbR^{n+1}_+}
&=\int_{\p\mcW^F\cap \mbR^{n+1}_+}F(\nu(y))\rd A(y)
\le \int_{\widetilde{\S_+}}F(\nu(p))H_n^F(p) \rd A(p)
\\&\le \int_{\widetilde{\S_+}}F(\nu)\left({H^F}\right)^n \rd A
\leq\int_\S F(\nu)\abs{{H^F}}^n \rd A,
}
which is \eqref{ineq-willmore-aniso-halfspace}.

If equality in \eqref{ineq-willmore-aniso-halfspace} holds, then all the inequalities in the above argument are actually equalities, so we readily infer that $\S$ agrees with $\widetilde{\S_+}$ 
and is in fact an anisotropic umbilical hypersurface, thanks to the AM-GM inequality. It is not hard to deduce that $\left<\nu_F(x),-E_{n+1}\right>=0$.
Hence $\S$ is an anisotropic free boundary Wulff shape in $\overline{\mbR^{n+1}_+}$.
\end{proof}

\subsection{Capillary hypersurfaces in a half-space}

As corollaries of Theorem \ref{Thm-aniso-halfspace}, we may deduce various Willmore-type inequalities for anisotropic capillary hypersurfaces in a half-space.

First, through the new Minkowski norm $\tilde{F}$ in \eqref{defn-tilde-F}  (see Proposition \ref{Prop-tilde-F}), the anisotropic capillary problem can be rephrased as an anisotropic free boundary problem.
Second, a special Minkowski norm  to our interest is defined as
\eq{\label{defn-F-theta}
F(\xi)=\abs{\xi}-\cos\theta_0\left<\xi,E_{n+1}\right>,
}
for some angle constant $\theta_0\in(0,\pi)$, by virtue of which the capillary problem in a half-space could be interpreted in terms of anisotropic terminology (see e.g., \cites{DeMasi22,LXX23}).
Plugging these Minkowski norms into Theorem \ref{Thm-aniso-halfspace}, we thereby obtain the following Willmore inequalities as claimed in the introduction, which
we state for readers' convenience.

\begin{theorem}\label{Thm-Willmore-ani-halfspace-Ch4}
Given $\om_0\in\left(-F(E_{n+1}),F(-E_{n+1})\right)$.
Let  $\S\subset\overline{\mbR^{n+1}_+}$ be a compact, embedded, $C^2$-hypersurface with boundary $\p\S\subset \p\mbR_+^{n+1}$ intersecting $\p\mbR^{n+1}_+$ transversally such that
\eq{
  \<\nu_F(x), -E_{n+1}\>=\om(x)\geq \om_0,\quad  \hbox{ for  any }x\in \partial\Sigma.
}
Then there holds
\eq{\label{ineq-willmore-ani-halfspace}
\frac1{n+1}\int_\S \left(F(\nu)+\om_0\left<\nu,E^F_{n+1}\right>\right)\abs{{H^F}}^n\rd A\ge \abs{\mcW^F_{1,\om_0}},
}
where $\mcW^F_{1,\om_0}=\mcW^F_1(\om_0 E^F_{n+1})\cap{\mbR_+^{n+1}}$.
Moreover, equality in \eqref{ineq-willmore-ani-halfspace} holds if and only if $\S$ is an anisotropic $\om_0$-capillary Wulff shape in $\overline{\mbR^{n+1}_+}$.
\end{theorem}

\begin{theorem}
Given $\theta_0\in(0,\pi)$.
Let  $\S\subset\overline{\mbR_+^{n+1}}$ be a compact, embedded $C^2$-hypersurface with boundary $\p\S\subset \p\mbR_+^{n+1}$ intersecting $\p\mbR^{n+1}_+$ transversally, such that
\eq{
\left<\nu, -E_{n+1}\right>\geq -\cos\theta_0,\quad \hbox{ for any } x\in\p\S.
}
Then there holds
\eq{\label{ineq-willmore-isotropic-halfspace}
\frac{1}{n+1}\int_{\Sigma}\left(1-\cos\theta_0\<\nu, E_{n+1}\>\right)\Abs{{H}}^{n} \rd A\ge \abs{B_{1,\theta_0}},
}
where  
$B_{1,\theta_0}=B_1(-\cos\theta_0 E_{n+1})\cap{\mbR_+^{n+1}}$.
Moreover, equality in \eqref{ineq-willmore-isotropic-halfspace} holds if and only if $\S$ is a $\theta_0$-capillary spherical cap in $\overline{\mbR^{n+1}_+}$.
\end{theorem}

On the other hand, by revisiting the geometric property concerning Gauss map in the (anisotropic) capillary settings, we could prove a variant of Theorem \ref{Thm-Willmore-ani-halfspace-Ch4} as follows.
\begin{theorem}\label{Thm-aniso-halfspace-'}
Given $\om_0\in\left(-F(E_{n+1}),F(-E_{n+1})\right)$.
Let  $\S\subset\overline{\mbR^{n+1}_+}$ be a compact, embedded, $C^2$-hypersurface with boundary $\p\S\subset \p\mbR_+^{n+1}$ intersecting $\p\mbR^{n+1}_+$ transversally such that
\eq{\label{condi-aniso-half-space-'}
  \<\nu_F(x), -E_{n+1}\>=\om(x)\geq \om_0,\quad  \hbox{ for  any }x\in \partial\Sigma.
}
Then there holds
\eq{\label{ineq-willmore-aniso-halfspace-'}
\int_\S F(\nu)\abs{{H^F}}^n \rd A\ge \int_{\p\mcW^F_{1,\om_0}\cap\mbR^{n+1}_+}F(\nu)\rd A.
}
Moreover, equality in \eqref{ineq-willmore-aniso-halfspace-'} holds if and only if $\S$ is an anisotropic $\om_0$-capillary Wulff shape in $\overline{\mbR^{n+1}_+}$.
\end{theorem}
The revisited variant of Proposition \ref{Prop-Gauss-map} reads as follows.
\begin{proposition}\label{Prop-Gauss-map'}
Given $\om_0\in\left(-F(E_{n+1}),F(-E_{n+1})\right)$.
Let $\S\subset\overline{\mbR^{n+1}_+}$ be a compact, embedded, $C^2$-hypersurface with boundary $\p\S\subset\p\mbR_+^{n+1}$ such that
\eq{
  \<\nu_F(x), -E_{n+1}\>=\om(x)\geq\om_0,\quad  \hbox{ for  any }x\in \partial\Sigma.
}
Then the anisotropic Gauss map $\nu_F:\S\ra\p\mcW^F$ satisfies
\eq{
\{y\in\p\mcW^F:\<y, -E_{n+1}\>\leq\om_0\}
\subset\nu_F(\widetilde{\S_+}).
}
\end{proposition}

Proposition \ref{Prop-Gauss-map'} follows from  Proposition \ref{Prop-Gauss-map} by introducing $\tilde{F}$ in \eqref{defn-tilde-F} and using Proposition \ref{Prop-tilde-F}.

\begin{proof}[Proof of Theorem \ref{Thm-aniso-halfspace-'}]
We note that
\eq{
\left\{y\in\p\mcW^F:\left<y,-E_{n+1}\right>\leq\om_0\right\}+\om_0E_{n+1}^F
=\p\mcW^F_1(\om_0 E^F_{n+1})\cap{\mbR_+^{n+1}},
}
and hence $$\int_{\left\{y\in\p\mcW^F:\left<y,-E_{n+1}\right>\leq\om_0\right\}}F(\nu(y))\rd A(y)=\int_{\p\mcW_{1,\om_0}^F\cap \mbR^{n+1}_+}F(\nu(y))\rd A(y).$$
By virtue of  Proposition \ref{Prop-Gauss-map'},
 we use the area formula in the following way,
\eq{
&\int_{\p\mcW_{1,\om_0}^F\cap \mbR^{n+1}_+}F(\nu(y))\rd A(y)
\le \int_{\widetilde{\S_+}}F(\nu(p))H_n^F(p) \rd A(p)
\\&\le \int_{\widetilde{\S_+}}F(\nu)\left({H^F}\right)^n \rd A
\leq\int_\S F(\nu)\abs{{H^F}}^n \rd A,
}
which is  \eqref{ineq-willmore-aniso-halfspace-'}.

The proof of rigidity follows similarly as that in the proof of Theorems \ref{Thm-aniso-halfspace}.
\end{proof}

\begin{remark}
\normalfont
In view of \eqref{eq2.9}, it is clear that the two Willmore-type inequalities \eqref{ineq-willmore-ani-halfspace} and \eqref{ineq-willmore-aniso-halfspace-'} are different unless $\om_0=0$.
\end{remark}


\bibliographystyle{alpha}
\bibliography{BibTemplate.bib}

\begin{bibdiv}
\begin{biblist}

\bib{AFM20}{article}{
      author={Agostiniani, Virginia},
      author={Fogagnolo, Mattia},
      author={Mazzieri, Lorenzo},
       title={\href{https://doi.org/10.1007/s00222-020-00985-4}{{S}harp
  geometric inequalities for closed hypersurfaces in manifolds with nonnegative
  {R}icci curvature}},
        date={2020},
        ISSN={0020-9910,1432-1297},
     journal={Invent. Math.},
      volume={222},
      number={3},
       pages={1033\ndash 1101},
         url={https://doi.org/10.1007/s00222-020-00985-4},
      review={\MR{4169055}},
}

\bib{Chen71}{article}{
      author={Chen, Bang-yen},
       title={\href{https://doi.org/10.1007/BF01351818}{On a theorem of
  {F}enchel-{B}orsuk-{W}illmore-{C}hern-{L}ashof}},
        date={1971},
        ISSN={0025-5831,1432-1807},
     journal={Math. Ann.},
      volume={194},
       pages={19\ndash 26},
         url={https://doi.org/10.1007/BF01351818},
      review={\MR{291994}},
}

\bib{CGR06}{article}{
      author={Choe, Jaigyoung},
      author={Ghomi, Mohammad},
      author={Ritor\'{e}, Manuel},
       title={\href{http://projecteuclid.org/euclid.jdg/1143593128}{{T}otal
  positive curvature of hypersurfaces with convex boundary}},
        date={2006},
        ISSN={0022-040X,1945-743X},
     journal={J. Differential Geom.},
      volume={72},
      number={1},
       pages={129\ndash 147},
         url={http://projecteuclid.org/euclid.jdg/1143593128},
      review={\MR{2215458}},
}

\bib{CGR07}{article}{
      author={Choe, Jaigyoung},
      author={Ghomi, Mohammad},
      author={Ritor\'{e}, Manuel},
       title={\href{https://doi.org/10.1007/s00526-006-0027-z}{{T}he relative
  isoperimetric inequality outside convex domains in {${\bf R}^n$}}},
        date={2007},
        ISSN={0944-2669,1432-0835},
     journal={Calc. Var. Partial Differential Equations},
      volume={29},
      number={4},
       pages={421\ndash 429},
         url={https://doi.org/10.1007/s00526-006-0027-z},
      review={\MR{2329803}},
}

\bib{DeMasi22}{misc}{
      author={De~Masi, Luigi},
       title={Existence and properties of minimal surfaces and varifolds with
  contact angle conditions},
        date={2022},
        note={Thesis (Ph.D.)--SISSA,
  \href{http://hdl.handle.net/20.500.11767/129590}{url}},
}

\bib{DePM15}{article}{
      author={De~Philippis, G.},
      author={Maggi, F.},
       title={\href{https://doi.org/10.1007/s00205-014-0813-2}{Regularity of
  free boundaries in anisotropic capillarity problems and the validity of
  {Y}oung's law}},
        date={2015},
        ISSN={0003-9527,1432-0673},
     journal={Arch. Ration. Mech. Anal.},
      volume={216},
      number={2},
       pages={473\ndash 568},
         url={https://doi.org/10.1007/s00205-014-0813-2},
      review={\MR{3317808}},
}

\bib{FM23}{article}{
      author={Fusco, N.},
      author={Morini, M.},
       title={\href{https://doi.org/10.1007/s00526-023-02438-1}{Total positive
  curvature and the equality case in the relative isoperimetric inequality
  outside convex domains}},
        date={2023},
        ISSN={0944-2669,1432-0835},
     journal={Calc. Var. Partial Differential Equations},
      volume={62},
      number={3},
       pages={Paper No. 102, 32},
         url={https://doi.org/10.1007/s00526-023-02438-1},
      review={\MR{4550406}},
}

\bib{HK78}{article}{
      author={Heintze, Ernst},
      author={Karcher, Hermann},
       title={\href{http://www.numdam.org/item?id=ASENS_1978_4_11_4_451_0}{{A}
  general comparison theorem with applications to volume estimates for
  submanifolds}},
        date={1978},
        ISSN={0012-9593},
     journal={Ann. Sci. \'{E}cole Norm. Sup. (4)},
      volume={11},
      number={4},
       pages={451\ndash 470},
         url={http://www.numdam.org/item?id=ASENS_1978_4_11_4_451_0},
      review={\MR{533065}},
}

\bib{JWXZ22}{misc}{
      author={{Jia}, Xiaohan},
      author={{Wang}, Guofang},
      author={{Xia}, Chao},
      author={{Zhang}, Xuwen},
       title={"{H}eintze-{K}archer inequality and capillary hypersurfaces in a
  wedge"},
        date={2022},
        note={to appear in {\bf Ann. Sc. Norm. Super. Pisa Cl. Sci.} doi:
  \href{https://doi.org/10.2422/2036-2145.202212_001}{10.2422/2036-2145.202212\_001},
  \href{https://arxiv.org/abs/2209.13839}{arXiv:2209.13839}},
}

\bib{JWXZ23}{article}{
      author={Jia, Xiaohan},
      author={Wang, Guofang},
      author={Xia, Chao},
      author={Zhang, Xuwen},
       title={\href{https://doi.org/10.1007/s00205-023-01861-0}{{A}lexandrov's
  theorem for anisotropic capillary hypersurfaces in the half-space}},
        date={2023},
        ISSN={0003-9527,1432-0673},
     journal={Arch. Ration. Mech. Anal.},
      volume={247},
      number={2},
       pages={Paper No. 25, 19},
         url={https://doi.org/10.1007/s00205-023-01861-0},
      review={\MR{4562813}},
}

\bib{JWXZ23b}{article}{
      author={Jia, Xiaohan},
      author={Wang, Guofang},
      author={Xia, Chao},
      author={Zhang, Xuwen},
       title={Heintze-{K}archer inequality for anisotropic free boundary
  hypersurfaces in convex domains},
        date={2024},
        ISSN={2096-9856,2617-8702},
     journal={J. Math. Study},
      volume={57},
      number={3},
       pages={243\ndash 258},
      review={\MR{4817296}},
}

\bib{LRV22}{article}{
      author={Leonardi, Gian~Paolo},
      author={Ritor\'{e}, Manuel},
      author={Vernadakis, Efstratios},
       title={\href{https://doi.org/10.1090/memo/1354}{{I}soperimetric
  inequalities in unbounded convex bodies}},
        date={2022},
        ISSN={0065-9266,1947-6221},
     journal={Mem. Amer. Math. Soc.},
      volume={276},
      number={1354},
       pages={1\ndash 86},
         url={https://doi.org/10.1090/memo/1354},
      review={\MR{4387775}},
}

\bib{LWW23}{article}{
      author={Liu, Lei},
      author={Wang, Guofang},
      author={Weng, Liangjun},
       title={\href{https://doi.org/10.1016/j.jfa.2023.109945}{The relative
  isoperimetric inequality for minimal submanifolds with free boundary in the
  {E}uclidean space}},
        date={2023},
        ISSN={0022-1236,1096-0783},
     journal={J. Funct. Anal.},
      volume={285},
      number={2},
       pages={Paper No. 109945, 22},
         url={https://doi.org/10.1016/j.jfa.2023.109945},
      review={\MR{4571870}},
}

\bib{LXX23}{article}{
      author={Lu, Zheng},
      author={Xia, Chao},
      author={Zhang, Xuwen},
       title={\href{https://doi.org/10.1515/ans-2022-0078}{{C}apillary
  {S}chwarz symmetrization in the half-space}},
        date={2023},
        ISSN={1536-1365,2169-0375},
     journal={Adv. Nonlinear Stud.},
      volume={23},
      number={1},
       pages={Paper No. 20220078, 14},
         url={https://doi.org/10.1515/ans-2022-0078},
      review={\MR{4604661}},
}

\bib{Mag12}{book}{
      author={Maggi, Francesco},
       title={Sets of finite perimeter and geometric variational problems},
      series={Cambridge Studies in Advanced Mathematics},
   publisher={Cambridge University Press, Cambridge},
        date={2012},
      volume={135},
        ISBN={978-1-107-02103-7},
         url={https://doi.org/10.1017/CBO9781139108133},
        note={An introduction to geometric measure theory},
      review={\MR{2976521}},
}

\bib{Simon83}{book}{
      author={Simon, Leon},
       title={Lectures on geometric measure theory},
      series={Proceedings of the Centre for Mathematical Analysis, Australian
  National University},
   publisher={Australian National University, Centre for Mathematical Analysis,
  Canberra},
        date={1983},
      volume={3},
        ISBN={0-86784-429-9},
      review={\MR{756417}},
}

\bib{Wang23}{article}{
      author={Wang, Xiaodong},
       title={\href{https://doi.org/https://doi.org/10.5802/afst.1733}{{R}emark
  on an inequality for closed hypersurfaces in complete manifolds with
  nonnegative {R}icci curvature}},
        date={2023},
        ISSN={0240-2963,2258-7519},
     journal={Ann. Fac. Sci. Toulouse Math. (6)},
      volume={32},
      number={1},
       pages={173\ndash 178},
      review={\MR{4574743}},
}

\bib{Willmore68}{article}{
      author={Willmore, T.~J.},
       title={Mean curvature of immersed surfaces},
        date={1968},
        ISSN={0041-9109},
     journal={An. \c{S}ti. Univ. ``Al. I. Cuza'' Ia\c{s}i Sec\c{t}. I a Mat.
  (N.S.)},
      volume={14},
       pages={99\ndash 103},
      review={\MR{238220}},
}

\bib{Winklmann06}{article}{
      author={Winklmann, Sven},
       title={\href{https://doi.org/10.1007/s00013-006-1685-y}{A note on the
  stability of the {W}ulff shape}},
        date={2006},
        ISSN={0003-889X,1420-8938},
     journal={Arch. Math. (Basel)},
      volume={87},
      number={3},
       pages={272\ndash 279},
         url={https://doi.org/10.1007/s00013-006-1685-y},
      review={\MR{2258927}},
}

\end{biblist}
\end{bibdiv}

\end{document}